\documentclass[10pt,a4paper]{amsart}
\usepackage{amsmath, amssymb, graphicx, subfigure, epsfig}
\usepackage{xcolor}
\numberwithin{equation}{section}

\newcommand{\tal}{\tilde\alpha}

\newcommand{\e}{\epsilon}
\newcommand{\ga}{\gamma}
\newcommand{\br}{\mathbb{R}}
\newcommand{\ik}{\varphi}
\newcommand{\pa}{\partial}

\newcommand{\al}{\alpha}
\newcommand{\avi}{\alpha_{\vec i}}

\newcommand{\om}{\omega}

\newcommand{\de}{\delta}
\newcommand{\fj}{\rho_0}
\newcommand{\xh}{x_h}
\newcommand{\be}{\begin{equation}}
\newcommand{\ee}{\end{equation}}
\newcommand{\kod}{\varphi}
\newcommand{\ts}{\text{supp}}
\newcommand{\mHo}{{\mathcal H}_0}
\newcommand{\dq}{\Delta t_1}
\newcommand{\cp}{\check p}
\newcommand{\dt}{\text{det}}

\newtheorem{theorem}{Theorem}
\newtheorem{lemma}{Lemma}
\newtheorem{corollary}{Corollary}
\newtheorem{definition}{Definition}

\begin{document}

\title[Reconstruction from discrete data in $\br^3$]{Analysis of reconstruction from discrete Radon transform data in $\br^3$ when the function has jump discontinuities}
\author[A Katsevich]{Alexander Katsevich$^1$}
\thanks{$^1$Department of Mathematics, University of Central Florida, Orlando, FL 32816.\\ 
This work was supported in part by NSF grant DMS-1615124.}

\begin{abstract}
In this paper we study reconstruction of a function $f$ from its discrete Radon transform data in $\br^3$ when $f$ has jump discontinuities.  
Consider a conventional parametrization of the Radon data in terms of the affine and angular variables. The step-size along the affine variable is $\e$, and the density of measured directions on the unit sphere is $O(\e^2)$. Let $f_\e$ denote the result of reconstruction from the discrete data. Pick any generic point $x_0$ (i.e., satisfying some mild conditions), where $f$ has a jump. Our first result is an explicit leading term behavior of $f_\e$ in an $O(\e)$-neighborhood of $x_0$ as $\e\to0$. 

A closely related question is why can we accurately reconstruct functions with discontinuities at all? This is a fundamental question, which has not been studied in the literature in dimensions three and higher. We prove that the discrete inversion formula ``works'', i.e. if $x_0\not\in S:=\text{singsupp}(f)$ is generic, then $f_\e(x_0)\to f(x_0)$ as $\e\to0$. The proof of this result reveals a surprising connection with the theory of uniform distribution (u.d.). This is a new phenomenon that has not been known previously. We also present some numerical experiments, which confirm the validity of the developed theory.
\end{abstract}
\maketitle

\section{Introduction}\label{sec_intro}

In this paper we consider the question: ``Why and how well does tomographic reconstruction from discrete data work?'' The answer is known if the function $f$ to be reconstructed is sufficiently smooth. Related problems are usually studied by sampling theory, where the goal is to estimate how closely the reconstructed function $f_\e$ approximates $f$ in some global norm \cite{kr89, des93, nat93, nat95, pal95, cb97, fr00, far04,  rs07, izen12}. Usually, $f$ is required to be essentially bandlimited, which imposes a smoothness requirement on $f$. In the case when $f$ is not smooth (e.g. has jump discontinuities) the question has not been studied much. It consists of two parts:
\begin{itemize}
\item[Q1.] What does reconstruction look like near the singularities of $f$?
\item[Q2.] What is the effect of ``remote'' singularities, i.e. the part of $\text{singsupp}(f)$ located at a distance from the reconstruction point? 
\end{itemize}

Frequently, one is less interested in how $f_\e$ approximates $f$ in some global norm. Instead, one would like to know how accurately and with what resolution the singularities of $f$ are reconstructed. Convergence of reconstruction algorithms in the case of objects with discontinuities has been studied as well \cite{gonchar86,  popov90, pal93, popov98}. However, in these works the discontinuities of the object are a complicating factor rather than the object of the study. The first paper focusing specifically on the behavior of $f_\e$ near a jump discontinuity of $f$ is \cite{kat_2017}, where the author considers inversion of the Radon transform in $\br^2$ in the parallel beam setting. The object can be static or change with time. The parametrization of the data is conventional, i.e. in terms of the affine and angular variables. Suppose the step-sizes along the angular and affine variables are $O(\e)$. Let $f_\e$ denote the result of reconstruction from the descrete data. Pick a point $x_0$, where $f$ has a jump. We suppose that $x_0$ is {\it generic}, i.e. it satisfies some mild conditions. The result of \cite{kat_2017} is an explicit leading term behavior of $f_\e$ in an $O(\e)$-neighborhood of $x_0$ as $\e\to0$. The obtained behavior, which we call {\it transition behavior} or, equivalently, {\it edge response} provides the desired resolution of the reconstruction algorithm. No such results were known in dimensions greater than two. In this paper we obtain the edge response in $\br^3$.

A closely related question is why can we accurately reconstruct functions with discontinuities at all? This is a fundamental question, which has not been studied in the literature in dimensions three and higher. The answer to the question is complicated. In $\mathbb R^2$, convergence of reconstruction algorithms in the case of objects with discontinuities has been studied in   \cite{gonchar86,  pal90, popov90, popov98}. In this paper we answer the question in $\br^3$. The key reason why reconstruction works comes from a surprising connection with the theory of uniform distribution (u.d.). This is a completely new phenomenon, which has not been noticed previously. The first hint of such a connection appeared in \cite{kat_2017}, where the transition behavior is derived using the u.d. property of a relevant sequence. 

Let us discuss the main results of this paper in more detail. Similarly to the 2D case, we assume here that the step-size along the affine variable is $\e$, and the density of measured directions on the unit sphere is $O(\e^2)$. First, we extend the notion of a generic point. Second, we obtain the transition behavior of $f_\e$ in an $O(\e)$-neighborhood of a generic point $x_0$ where $f$ has a jump discontinuity. This answers the question Q1 above. A connection with the u.d. theory here is similar to the one in \cite{kat_2017}. 

Third, we answer the question Q2. We show that if $x_0\not\in S:=\text{singsupp}(f)$ is generic, then $f_\e(x_0)\to f(x_0)$ as $\e\to0$. This means that if $x_0\in S$, then the singularities of $f$ at a distance from $x_0$ (i.e., ``remote'' singularities) do not contribute to the transition behavior at $x_0$. Only the behavior of $f$ near $x_0$ contributes to the edge response. There is no contradiction between the assumptions $x_0\in S$ and $x_0\not\in S$, since the singularities of $f$ near $x_0$ and away from $x_0$ can be separated by a partition of unity using the linearity of the Radon transform. Additionally, this result proves that the discrete inversion formula ``works'' (i.e., provides increasingly accurate results as $\e\to0$) and illuminates the reasons why. We show that reconstruction is guaranteed to work if certain collection of points is u.d. inside a set, whose shape depends on $S$. A major complication in the proof is that the set shrinks as $\e\to0$. This connection with the u.d. theory never appeared in the literature before (including \cite{kat_2017}).

As was noted in \cite{pal90, pal93}, the behavior of the reconstruction algorithm depends on the geometry of $S$. Clearly, it also depends on the strength of singularities of $f$ (e.g., in the Sobolev scale). Comprehensive analysis of all these cases is beyond the scope of this paper. Here we consider only the case when $f$ has jump discontinuities across $S$, and $S$ is the union of smooth convex surfaces. 

The paper is organized as follows. In Section~\ref{crt_3D} we describe the problem set-up, specify the class of functions we consider, describe discrete Radon transform data, introduce the discrete inversion formula, define the notion of a generic point, and formulate the main result (Theorem~\ref{thm:main}). In Section~\ref{sec:local} we obtain the edge response in a neighborhood of a generic point $x_0\in S$. In Section~\ref{remote-crt} we show that ``remote'' singularities do not contribute to the edge response. In Section~\ref{numerix} we illustrate on numerical experiments that when $x_0$ is generic, then the theoretically predicted and numerically computed edge responses match very well. We also show that if $x_0$ is not generic, then the match may no longer be accurate. Technical results related to the u.d. property of certain sets of points, which are needed in Section~\ref{remote-crt}, are proven in Appendix~\ref{sex:auxres}.

\section{Preliminary construction}\label{crt_3D}

Consider a function $f(x)$, which can be represented as a finite sum
\begin{equation}\label{f_def_1}
f(x)=\sum_j \chi_{D_j} f_j(x),
\end{equation}
where $\chi_{D_j}$ is the characteristic function of the domain $D_j$. We assume that for each $j$:
\begin{enumerate}
\item $D_j$ is bounded,  
\item The boundary of $D_j$ is $C^{\infty}$ and convex,  
\item $f_j$ is $C^\infty$ in a domain containing the closure of $D_j$.
\end{enumerate}
The Radon transform of $f$ is defined as follows:
\be\label{crt_1}
g(\al,p)=\int _{\al\cdot x=p} f(x)dx,
\ee
where $dx$ is the area element on the plane $\Pi(\al,p):=\{x\in\br^3:\al\cdot x=p\}$. The discrete data are given by
\be\label{data_crt}
g(\al_i,p_j),\ p_j=\e(\cp(\al_i)+j),\ j\in\mathbb Z,
\ee
where $\cp\in C^1(S^2)$, and $i\in\mathbb N$, $i\le O(\e^{-2})$, is the index that enumerates the measured directions $\al_i\in S^2$. When considering $\al_i$ in some small open set $\Omega\subset S^2$, we will use a 2D multi-index $\vec i=(i_1,i_2)$ instead of $i$: $\al_{\vec i}$. The assumption is that for almost all $\al_0\in S^2$ there exists an open set $\Omega$, $\al_0\in\Omega\subset S^2$, such that 
\be\label{alH}
\avi=H(\e(i_1+r_1),\e(i_2+r_2)) \text{ for any }\avi\in\Omega. 
\ee
Here $r_{1,2}$ are some constants, $0\le r_{1,2}<1$, which may depend on $\e$. The function $H(\vec t):U\to \Omega$, where $U\subset\br^2$ is some bounded domain, is a smooth diffeomorphism. Both $H$ and $U$ may depend on $\al_0$ and $\Omega$. We also assume that the determinant 
\be\label{dHnz}
\dt(H'(\vec t))=\dt\left(\frac{\pa\al}{\pa (t_1,t_2)}\right)>0
\ee
is bounded away from zero in each $U$.

Let $\varphi$ be an interpolating kernel, which satisfies the following assumptions:
\begin{itemize}
\item[A1.] $\varphi$ is exact up to the order $2$, i.e.
\be\label{ker-int}
\sum_{j\in \mathbb Z} j^m\varphi(t-j)=t^m,\quad 0\le m \le 2,\ t\in\br,
\ee
\item[A2.] $\varphi$ is compactly supported, 
\item[A3.] The derivatives $\varphi^{(m)}$, $1\le m\le 3$, exist, 
\item[A4.] $\varphi^{(3)}$ is piecewise continuous and bounded,
\item[A5.] \label{normaliz} $\int\varphi(t)dt=1$.
\end{itemize}
The interpolated (in $p$) version of $g$ can be written in the form
\be\label{f-int-crt}
g_\e(\al_i,p):=\sum_{j \in \mathbb Z} g(\al_i,\e(\cp(\al_i)+j)) \varphi\left(\frac{p-\e(\cp(\al_i)+j)}{\e}\right).
\ee
To simplify notations, the dependence of $\cp$ on $\al$ is omitted in most places.

Since the Radon transform is even, the classical inversion formula with continuous data is given by
\be\label{invf-cont-crt}
f(x)=-\frac1{4\pi^2}\int_{S_+^2}\left. (\pa/\pa p)^2 g(\al,p)\right|_{p=\al\cdot x} d\al,
\ee
where $S_+^2$ is any half of the unit sphere in $\br^3$. 
The discrete inversion formula is given by 
\be\label{invf-discr-crt}
f_\e(x)=-\frac1{4\pi^2}\sum_{i:\,\al_i\in S_+^2} c_i\left. (\pa/\pa p)^2 g_\e(\al_i,p)\right|_{p=\al_i\cdot x},
\ee
where $c_i$ are integration weights. We assume that there exist $0<a\le b$ so that
\be\label{wghts}
a\e^2\le c_i \le b\e^2
\ee
for all $i$ as $\e\to0$. A convenient and useful way to think of the coefficients $c_i$ is that they represent the area of pieces that tessellate $S^2$, each piece is sufficiently regular and contains only one $\al_i$. In particular, $c_i=|H'(\vec t)|\e^2+O(\e^3)$ for any $\vec t=(t_1,t_2)$ such that $\al=H(\vec t)$ belongs to the same tessellation piece as $\al_i$. 

Pick a point $x_0\in S:=\text{singsupp} f$, and suppose that $S$ is a smooth surface with positive principal curvatures in a neighborhood of $x_0$. Let $\Theta_0$ be the unit vector normal to $S$ at $x_0$ and pointing into the interior of the corresponding domain $D_j$.
Consider the point
\be\label{rec-pt-crt}
\xh:=x_0+\e h\Theta_0,
\ee
where $h$ varies over a bounded set. Denote also:
\be\label{f0-def-crt}
f_0:=\lim_{\e\to 0^+}f(x_0+\e\Theta_0),\ \fj:=\lim_{\e\to 0^+}(f(x_0+\e\Theta_0)-f(x_0-\e\Theta_0)).
\ee
Introduce the coordinates $\al^\perp$ on $S^2$ in a neighborhood of $\Theta_0$:
\be\label{rotate-crt}
\al=\al^\perp+\al_3\Theta_0\in S^2,\ \al^\perp=(\al_1,\al_2)\in\Theta_0^\perp,\ \al_3=\sqrt{1-|\al^\perp|^2}.
\ee

Let $p=p_0(\al):\, S^2\to\br$ be a function determined by the condition that the plane $\Pi(\al,p_0(\al))$ be tangent to $S$. Generally, $p_0$ can be multi-valued. In what follows we always consider one of its single-valued local branches. Pick some $x_0$ and introduce the functions 
\be\label{new-vars-orig}\begin{split}
q(\al):=\al\cdot x_0,\ v(\al):=\al\cdot x_0-p_0(\al).
\end{split}
\ee
Pick any $\al_0\in S^2$ such that $x_0\in \Pi(\al_0,p_0(\al_0))$. If $S$ has positive principal curvatures at a point $z_0$, where $\Pi(\al_0,p_0(\al_0))$ is tangent to $S$, and $x_0\not=z_0$, then the equation $v(\al)=0$ determines a locally smooth curve, which we denote by $\Gamma$. By construction, $\al_0\in\Gamma$. The curve is smooth, because $v'(\al_0)=x_0^\perp-z_0^\perp\not=0$. The derivative with respect to $\al$ is computed on the unit sphere, and $x_0^\perp,z_0^\perp$ are the orthogonal projections of $x_0,z_0$, respectively, onto the plane $\al_0^\perp$. Clearly, $\Pi(\al_0,p_0(\al_0))$ can be tangent to $S$ at other points, but here $p_0(\al)$ is the local branch that is valid in a neighborhood of $z_0$. Our statements apply to any of the branches and the corresponding points of tangency.

\begin{definition} We say that a point $x_0\in S=\text{singsupp} f$ is generic if
\begin{enumerate}
\item $S$ is a smooth surface with positive principal curvatures in a neighborhood of $x_0$. In particular, $S$ does not self-intersect at $x_0$;
\item\label{Hgrad} Let $\Theta_0$ be the unit vector normal to $S$ at $x_0$, and let $\vec t^*$ be such that $\Theta_0=H(t^*)$. The vector $\nabla (H(\vec t)\cdot x_0)|_{\vec t=\vec t^*}$ has at least one irrational component; and
\item\label{poscurv} If a plane $\Pi$ contains $x_0$ and is tangent to $S$ at some $z_0\not=x_0$, then the principal curvatures of $S$ at $z_0$ are positive.
\end{enumerate}
An additional condition for a point to be generic is formulated in Section~\ref{remote-crt}.
\end{definition}

\begin{definition} We say that a point $x_0\not\in S=\text{singsupp} f$ is generic if it satisfies condition \eqref{poscurv} above and the additional condition in Section~\ref{remote-crt}.
\end{definition}

Now we formulate the main result of the paper.
\begin{theorem}\label{thm:main} For a generic $x_0\in S$, one has
\be\label{main-res}
\lim_{\e\to0}f_\e(\xh)=f_0-\fj \int_h^{\infty} \varphi(s) d s.
\ee
\end{theorem}
The proof of the theorem is split into two sections.

\section{Contribution of the local singularity}\label{sec:local}

In this section we choose the hemisphere $S_+^2:=\{\al\in S^2:\ \al_3>0\}$ in \eqref{invf-cont-crt} and \eqref{invf-discr-crt}. The range of $\al$ in \eqref{invf-discr-crt} can be split into three sets:
\be\label{sets-crt} \begin{split}
\Omega_1&:=\{\al\in S_+^2:\,|\al^\perp|<A\sqrt\e\},\ \Omega_2:=\{\al\in S_+^2:\,A\sqrt\e<|\al^\perp|<\om\}, \\
\Omega_3&:=S_+^2\setminus (\Omega_1\cup\Omega_2),
\end{split}
\ee
for some small (but fixed) $\om>0$. Here $A>0$ is a large parameter. Let $f_\e^{(j)}$ denote the value of the sum in \eqref{invf-discr-crt} ranging over indices $i$ such that $\al_i\in\Omega_j$, $j=1,2,3$. Consider $f_\e^{(1)}$ first. 

It will be shown below (see Section~\ref{lot-crt}) that the Radon transform of locally smooth components of $f$ does not contribute to the transition behavior of $f_\e$. Hence, without loss of generality, we can assume that $f\equiv0$ in a neighborhood of $x_0$ on the exterior side of $S$ (this is the side corresponding to $h<0$ in \eqref{rec-pt-crt}). By linearity and using a partition of unity if necessary, we can assume that $\ts(f)$ is contained in a small neighborhood of $x_0$. The local branch of $p_0$ that we use in this section is determined by the condition that the point of tangency be close to $x_0$. As is known, \cite{rz2, rz1}, in this case we have
\be\label{loc-v2-crt}
g(\al,p)=(p-p_0(\al))_+G(\al,p-p_0(\al)),
\ee
where $G$ is a smooth function near $(\Theta_0,0)$,  
\be\label{G-val}
G(\al,0)=2\pi \fj(\al)/\sqrt{\text{det}Q(\al)},
\ee
and $Q(\al)$ is the matrix of the second fundamental form of $S$ in the basis $(\al_1,\al_2)$ at the point where $\Pi(\al,p_0(\al))$ is tangent to $S$. A simple calculation shows that
\be\label{t0-eq-crt}
p_0(\al)=\al\cdot x_0-\frac{Q^{-1}(\Theta_0)\al^\perp\cdot \al^\perp}{2}+O(\e^{3/2})
\ee
when $|\al^\perp|=O(\sqrt\e)$.
Hence the interpolated data are given by 
\be\label{lead-t-int-crt}\begin{split}
g_\e(\al_i,p)= & \sum_{j\in \mathbb Z}\left[\frac{2\pi \fj(\al_i)}{\sqrt{\text{det}Q(\al_i)}} \left(\e(\cp+j)-p_0(\al_i)\right)_+ +O\left((\e(\cp+j)-p_0(\al_i))_+^2\right)\right]\\
&\hspace{3cm}\times\varphi\left(\frac{p-\e(\cp+j)}{\e}\right).
\end{split}
\ee

In view of \eqref{invf-discr-crt} and \eqref{lead-t-int-crt}, denote
\be\label{psi-def-crt}
\psi(q,s):=\sum_{j\in \mathbb Z} (j-q+s)_+
\varphi''(q-j),\ q,s\in\br.
\ee
The assumptions A1--A4 imply that
\be\label{psi-ders}
\text{the derivatives $\psi'_q$ and $\psi'_s$ are piecewise-continuous and bounded},
\ee
\be\label{psi-per}
\psi(q,s)=\psi(q+n,s) \text{ for any }n\in\mathbb Z, 
\ee
\be\label{psi-loc}
\psi(q,s)=0 \text{ if $|s|$ is large enough.} 
\ee
Combining \eqref{t0-eq-crt}--\eqref{psi-def-crt} and using \eqref{invf-discr-crt}, \eqref{sets-crt}, we obtain
\be\label{f1-lt-crt}
\begin{split}
& f_\e^{(1)}(x_h)\\
&= -\frac{1}{2\pi} \sum_{|\al_i^\perp|<A\sqrt\e} c_i\sum_{j\in \mathbb Z}\\
&\qquad \biggl[\frac{\fj(\al_i)}{\sqrt{\text{det}Q(\al_i)}}\left(\e(\cp+j)-\al_i\cdot x_0+\frac{Q^{-1}\al_i^\perp\cdot \al_i^\perp}{2}+O(\e^{3/2})\right)_+
+O(\e^2)\biggr]\\
&\qquad\times\e^{-2}\ik''\left(\frac{\al_i\cdot x_0-\check p}{\e}+h+O(\e)-j\right)\\
&=-\frac{\fj}{2\pi\sqrt{\text{det}Q}} \sum_{|\al_i^\perp|<A\sqrt\e} \frac{c_i}{\e}
\psi\left(\frac{\al_i\cdot x_0}{\e}-\cp_0+h,h+\frac{Q^{-1}\al_i^\perp\cdot \al_i^\perp}{2\e}\right)+O(\e^{1/2}),
\end{split}
\ee
where we have introduced the notation
\be\label{f0Q-def}
\fj:=\fj(\Theta_0),\ \cp_0:=\cp(\Theta_0),\ Q:=Q(\Theta_0).
\ee
The $O(\e^2)$ term in brackets in \eqref{f1-lt-crt} corresponds to the second term in brackets in \eqref{lead-t-int-crt}. This follows because $|x_h-x_0|=O(\e)$, $\ik$ is compactly supported, and $|\al_i^\perp|=O(\e^{1/2})$ (cf. \eqref{sets-crt} and \eqref{t0-eq-crt}). In \eqref{f1-lt-crt} we used 
that (a) there are $O(1/\e)$ terms in the sum with respect to $i$, and (b) using \eqref{wghts}, the sum in \eqref{f1-lt-crt} remains bounded as $\e\to0$. 

Set $\tilde\al^\perp:=\al^\perp/\sqrt\e$. The tessellation of $S^2$ based on $\al_i$ induces a tessellation of the domain $|\tilde\al^\perp|<A$. Let $\tilde c_i$ denote the area of the tessellation piece corresponding to $\tilde\al_i^\perp$. From \eqref{rotate-crt},
\be\label{areas}
\tilde c_i=\frac{c_i}{\e}\left(1+O(\e^{1/2})\right).
\ee
The term $O(\e^{1/2})$ here also includes the Jacobian of the transformation $\br^2\ni \al^\perp\to\al\in S^2$.

Apply \eqref{alH} to an open set containing $\Theta_0$. By shifting indices $i_{1,2}$ and adjusting $r_{1,2}\in[0,1)$, if necessary, we can always assume that $H(0,0)=\Theta_0$. Denote 
\be\label{hprime}
\tilde q(\vec t):=H(\vec t)\cdot x_0,\ \mHo':= \tilde q'(\vec 0),\ \mHo'':= \tilde q''(\vec 0).
\ee
Then 
\be\label{frstarg}
\frac{\avi\cdot x_0}{\e}=\frac{\Theta_0\cdot x_0}{\e}+
(i_1+r_1,i_2+r_2)\cdot \mHo'+\frac{\mHo''\tilde\al_{\vec i}^\perp\cdot \tilde\al_{\vec i}^\perp}2+O(\e^{1/2}).
\ee
Combining the above, rewrite \eqref{f1-lt-crt} as follows
\be\label{f1-lt-crt-v3}
\begin{split}
f_\e^{(1)}(x_h)=& -\frac{\fj}{2\pi\sqrt{\text{det}Q}}J_\e+O(\e^{1/2}),\\
J_\e:= & \sum_{|\tilde\al_{\vec i}^\perp|<A} \tilde c_{\vec i}
\psi\left(A_\e+\vec i\cdot \mHo'+\frac{\mHo''\tilde\al_{\vec i}^\perp\cdot \tilde\al_{\vec i}^\perp}2,h+\frac{Q^{-1}\tilde\al_{\vec i}^\perp\cdot \tilde\al_{\vec i}^\perp}{2}\right),
\end{split}
\ee
where $A_\e$ depends on $\e$, but is independent of $\vec i$. The $O(\e^{1/2})$ term in \eqref{frstarg} is absorbed by the $O(\e^{1/2})$ term in \eqref{f1-lt-crt-v3}.

To find the limit of $J_\e$ as $\e\to0$, pick some small $\e'>0$ and break up the domain $|\tilde\al_{\vec i}^\perp|<A$ into squares $B_l$ of size $\e'$. 
Consider one $B_l$ and all $\tilde\al_{\vec i}^\perp$ within it. Replace all $\tilde\al_{\vec i}^\perp\in B_l$ with $\tilde\al_*^\perp$, where $\tilde\al_*^\perp$ is the center of $B_l$, everywhere in the arguments of $\psi$. This leads to an error of magnitude $O(\e')$. By $O(\e')$ we denote any quantity whose magnitude does not exceed $c\e'$, where $c$ is independent of the selected square $B_l$, the chosen $\tilde\al_{\vec i}^\perp\in B_l$, and $\e$. Clearly,
\be\label{psi-simple}
\begin{split}
&\psi\left(A_\e+\vec i\cdot \mHo'+\frac{\mHo''\tilde\al_{\vec i}^\perp\cdot \tilde\al_{\vec i}^\perp}2,h+\frac{Q^{-1}\tilde\al_{\vec i}^\perp\cdot \tilde\al_{\vec i}^\perp}{2}\right)\\
&=\psi\left(A_\e+\vec i\cdot \mHo'+\frac{\mHo''\tilde\al_*^\perp\cdot \tilde\al_*^\perp}2,h+\frac{Q^{-1}\tilde\al_*^\perp\cdot \tilde\al_*^\perp}{2}\right)+O(\e'),\ \tilde\al_{\vec i}^\perp\in B_l.
\end{split}
\ee
By \eqref{psi-per} and \eqref{psi-simple}, we can consider the set
\be\label{sqseq}
\left\{A_\e'+\vec i\cdot \mHo'\right\},\ A_\e'= A_\e+\frac{\mHo''\tilde\al_*^\perp\cdot \tilde\al_*^\perp}2,
\ee
where ${\vec i}$ are such that $\tilde\al_{\vec i}^\perp\in B_l$. There are $O(1/\e)$ such $\vec i$: $O(\e^{-1/2})$ per each coordinate $i_{1,2}$. Here and in what follows, $\{x\}$ denotes the fractional part of a number $x\in\br$.
By Theorem 2.9 and Example 2.9 of \cite{KN_06}, the set $\left\{\vec i\cdot \mHo' \right\}$ is u.d. as $\e\to0$ if at least one of the components of the vector $\mHo'$ is irrational (see condition (2) in the definition of a generic point). Clearly, the shift by $A_\e'$ does not affect the u.d. property. Therefore, for the sum of the terms in \eqref{f1-lt-crt-v3} corresponding to $\tilde\al_i^\perp\in B_l$ we obtain
\be\label{part-sum}
\begin{split}
&\sum_{\tilde\al_{\vec i}^\perp\in B_l} \tilde c_{\vec i}
\psi\left(A_\e+\vec i\cdot \mHo'+\frac{\mHo''\tilde\al_{\vec i}^\perp\cdot \tilde\al_{\vec i}^\perp}2,h+\frac{Q^{-1}\tilde\al_{\vec i}^\perp\cdot \tilde\al_{\vec i}^\perp}{2}\right)\\
&=\sum_{\tilde\al_{\vec i}^\perp\in B_l} \tilde c_{\vec i}
\psi\left(\left\{A_\e'+\vec i\cdot \mHo'\right\},h+\frac{Q^{-1}\tilde\al_*^\perp\cdot \tilde\al_*^\perp}{2}\right)+O(\e')\\
&=(\tilde c_*N_\e(B_l))\frac1{N_\e(B_l)}\sum_{\tilde\al_{\vec i}^\perp\in B_l} \psi\left(\left\{A_\e'+\vec i\cdot \mHo'\right\},h+\frac{Q^{-1}\tilde\al_*^\perp\cdot \tilde\al_*^\perp}{2}\right)+O(\e')\\
&\to \text{Area}(B_l)\left[\int_0^1 
\psi\left(t,h+\frac{Q^{-1}\tilde\al_*^\perp\cdot \tilde\al_*^\perp}{2}\right)dt+O(\e')\right],\
\e\to0.
\end{split}
\ee
Here $N_\e(B_l)$ is the number of terms in the sum in \eqref{part-sum}, and $\tilde c_*$ is the average of $\tilde c_{\vec i}$. By construction, $|\tilde c_{\vec i}-\tilde c_*|/\e=O(\e')$ whenever $\tilde\al_{\vec i}^\perp\in B_l$. 

Sum over all squares and observe that the resulting Riemann sum is within $O(\e')$ from the corresponding integral. This yields:
\be\label{jelim}
\begin{split}
\lim_{\e\to0} J_\e=\int_{|\tal^\perp|<A} \int_0^1
\psi\left(t,h+\frac{Q^{-1}\tal^\perp\cdot \tal^\perp}{2}\right)dt\,d\tal^\perp+O(\e').
\end{split}
\ee
Since $\e'>0$ can be as small as we like, combining \eqref{jelim} with \eqref{f1-lt-crt-v3} gives
\be\label{f1-st3-crt}
\lim_{\e\to0}f_\e^{(1)}(\xh)= -\frac{\fj}{2\pi\sqrt{\text{det}Q}} \int_{|\tal^\perp|<A} \int_0^1
\psi\left(t,h+\frac{Q^{-1}\tal^\perp\cdot \tal^\perp}{2}\right)dt\, d\tal^\perp.
\ee
In view of \eqref{psi-loc}, if $h$ is fixed, the integrand in \eqref{f1-st3-crt} is zero when $|\tal^\perp|$ is large enough (recall that $Q(\Theta_0)$ is positive definite). Hence we can select $A>0$ large enough and obtain by changing variables
\be\label{f1-st4-crt}
\begin{split}
\lim_{\e\to0}f_\e^{(1)}(\xh)&= -\frac{\fj}{2\pi\sqrt{\text{det}Q}} \int_{\br^2} \int_0^1
\psi\left(t,h+\frac{Q^{-1}\tal^\perp\cdot \tal^\perp}{2}\right)dt\,d\tal^\perp\\
&=-\fj \int_0^{\infty}\int_0^1 \psi(t,h+s)dt\,ds
=-\fj \int_h^{\infty}\int_0^1 \psi(t,s)dt\,ds.
\end{split}
\ee

\begin{lemma}\label{lem:U-lims-crt} One has 
\begin{equation}\label{U-lims-crt}
-\int_h^{\infty}\int_0^1 \psi(t,s)dt\,ds = \begin{cases}0,& h>c, \\
-1,& h< -c\end{cases} 
\end{equation}
for some $c>0$ large enough.
\end{lemma}
\begin{proof} By \eqref{psi-def-crt} and \eqref{psi-loc},
\be\label{int-psi-st1}
\begin{split}
\int_0^1 \psi(t,s)dt&= \sum_{j\in \mathbb Z}\int_0^{1}(j-t+s)_+
\varphi''(t-j)dt = \int_{\br}(s-t)_+\varphi''(t)dt= \varphi(s).
\end{split}
\end{equation}
Using A2 and A5 we finish the proof.
\end{proof}

\subsection{Analysis of the terms $f_\e^{(2)}$ and $f_\e^{(3)}$.}\label{lot-crt}

The Radon transform of $f$ is smooth when $(\al,p)$ are such that $\al\in\Omega_3$ and $|p-\al\cdot x_0|=O(\e)$, i.e. $\Pi(\al,p)$ is not tangent to $S$. Hence the interpolated data has derivatives with respect to $p$ up to the order $3$ uniformly bounded as $\e\to0$. Therefore,
\be\label{f3-crt-st1}\begin{split}
&|f_\e^{(3)}(\xh)-f_\e^{(3)}(x_0)|\\
& \le \frac1{4\pi^2}\sum_{i:\al_i\in\Omega_3} c_i\left|\left. (\pa/\pa p)^2 g_\e(\al_i,p)\right|_{p=\al_i\cdot\xh}
-\left. (\pa/\pa p)^2 g_\e(\al_i,p)\right|_{p=\al_i\cdot x_0}\right|.
\end{split}
\ee
For a function $g\in C^3(\br)$ we have
\be\label{third-der}\begin{split}
g_\e''(t+O(\e))-g_\e''(t)
&=\frac1{\e^2}\sum_j g(\e j)\left[\varphi''\left(\frac{t+O(\e)-\e j}{\e}\right) -
\varphi''\left(\frac{t-\e j}{\e}\right) \right]\\
&=\frac1{\e^2}\sum_j \left(g(t)+(\e j-t)g'(t)+(\e j-t)^2\frac{g''(t)}2+O(\e^3)\right)\\
&\hspace{1.5cm}\times\left[\varphi''\left(\frac{t+O(\e)}{\e}-j\right) -
\varphi''\left(\frac{t}{\e}-j\right) \right]\\
&=O(\e),
\end{split}
\ee
where we have used \eqref{ker-int}. Combining with \eqref{f3-crt-st1} this implies
\be\label{f3-crt-st2}
|f_\e^{(3)}(\xh)-f_\e^{(3)}(x_0)|=O(\e).
\ee

In a similar fashion, we will show that 
\be\label{part2-lim}
\lim_{A\to\infty}\lim_{\e\to0}|f_\e^{(2)}(\xh)-f_\e^{(2)}(x_0)|=0.
\ee
Choose $w>0$ in \eqref{sets-crt} so that $\al\cdot x_0-p_0(\al)$ is strictly convex on $\Omega_1\cup\Omega_2$. In view of \eqref{t0-eq-crt}, this is possible. Hence the lower bound on the values of $\al\cdot x_0-p_0(\al)$ on $\Omega_2$ is determined by looking at its values at the boundary of $\Omega_1$, where \eqref{t0-eq-crt} holds. Pick $c>0$ large enough. Our argument implies that there exist sufficiently large $A>0$ and sufficiently small $\e_0>0$ such that $\al\cdot x_0-p_0(\al)>c\e$ on $\Omega_2$ for all $\e$, $0<\e\le\e_0$. Since $\varphi$ is compactly supported, $g(\al,p)$ is a smooth function in a neighborhood of all $(\al_i,p_j)$ such that $\al_i\in\Omega_2$ and  $(\al_i\cdot \xh-p_j)/\e\in\ts(\varphi)$. The reason is that for such $(\al,p)$ we can drop the subscript $'+'$ from $(p-p_0)_+$ in \eqref{loc-v2-crt}:
\be\label{loc-v3-crt}
g(\al,p)=(p-p_0(\al))G(\al,p-p_0(\al)).
\ee
Hence \eqref{part2-lim} follows similarly to \eqref{f3-crt-st1}--\eqref{f3-crt-st2}.

Examining the above argument more carefully, we see that on $\Omega_1\cup\Omega_2$, the part of the Radon transform of $f$ that is used in the inversion formula coincides with a globally smooth function, so we have 
\be\label{twoterms}
\lim_{\e\to0}(f_\e^{(2)}(x_0)+f_\e^{(3)}(x_0))=-\frac1{4\pi^2}\int_{S_+^2}\left. (\pa/\pa p)^2 g(\al,p)\right|_{p=\al\cdot x_0+0} d\al=f_0,
\ee
where $f_0$ is defined in \eqref{f0-def-crt}. The notation ``$+0$'' means that the derivative with respect to $p$ is evaluated on the interior side of the domain where the Radon transform of $f$ is smooth. The last equality in \eqref{twoterms} is an immediate corollary of the inversion formula and the smoothness of the Radon transform. 

\section{Contribution of remote singularities}\label{remote-crt}

Suppose the plane $\Pi(\Theta_0,p_0(\Theta_0))$ is tangent to $S$ at some $z_0$, $z_0\not=x_0$. 
In this section $x_0$ is generic, but otherwise arbitrary, and $\Theta_0\in S^2$ is arbitrary as well. In particular, $x_0$ is not necessarily on $S$. Again, using linearity, we suppose that $\ts(f)$ is a subset of a small neighborhood of $z_0$. The surface $S$ can have a wide variety of shapes in a neighborhood of $z_0$. Some of these shapes may produce artifacts (e.g., if $S$ is locally flat near $z_0$), while others - will not. Comprehensive analysis of all such cases is beyond the scope of this paper. Here we will consider one generic situation: the principal curvatures of $S$ at $z_0$ are positive. In this section we prove that the singularity at $z_0$ does not contribute to the transition behavior of the reconstruction at $x_0$. More precisely, we will prove that in the limit $\e\to0$ the discretized reconstruction formula \eqref{invf-discr-crt} gives the same result as the exact reconstruction formula \eqref{invf-cont-crt}.

To illustrate the main ideas, we start with the case when the reconstruction point is $x_0$ rather than $\xh$. As noted in Section~\ref{crt_3D}, assumption (3) in the definition of a generic point implies that the equation $\al\cdot x_0=p_0(\al)$ determines a locally smooth curve denoted $\Gamma$.  Let $\chi(\al)$ be a smooth cut-off function with small support such that $\chi\not\equiv0$ on $\Gamma$.
Using \eqref{new-vars-orig}, introduce the functions 
\be\label{new-vars}\begin{split}
\tilde q(\vec t):=(q\circ H)(\vec t),\ \tilde v(\vec t):=(v\circ H)(\vec t):\, U\to\br,
\end{split}
\ee
where $U$ is the domain of $H$.
The function $\tilde q$ here is the same as in \eqref{hprime}.

Let $U_{t}$ and $\Gamma_{t}$ be the images of $\ts(\chi)$ and $\Gamma\cap\ts(\chi)$, respectively, under the map $H^{-1}:\,\al\to\vec t$. 
Using \eqref{dHnz} and that $v'(\al)\not=0$ for any $\al\in\Gamma_t$, we can assume that $\ts(\chi)$ is so small 
that either $\tilde v_1'(\vec t)\not=0$ for any $\vec t\in \Gamma_t$ or $\tilde v_2'(\vec t)\not=0$ for any $\vec t\in \Gamma_t$. Without loss of generality, suppose $\tilde v_2'(\vec t)\not=0$. This implies that an equation of $\Gamma_{t}$ can be written in the form
\be\label{tildeq}
\Gamma_{t}:\ t_2=A(t_1)
\ee
for some smooth $A$. Indeed, setting $\tilde v(t_1,A(t_1))\equiv0$, the Implicit Function theorem implies that $A(t_1)$ is well-defined, and $A'=-\tilde v_1'/\tilde v_2'$.

Now we can formulate the last condition that a generic point satisfies. 
{\it
\begin{enumerate}
\setcounter{enumi}{3}
\item \label{zeromeas} Pick any local piece $\Gamma_t$ constructed as above and any $M_1,M_2\in\mathbb Z$ such that $M_1^2+M_2^2>0$. Consider the function $f(t_1)=M_1\tilde q(t_1,A(t_1))+M_2 A(t_1)$ defined on the domain $\{t_1\in\br:\, (t_1,A(t_1))\in\Gamma_t\}$. We have
\begin{enumerate}
\item The set of $t_1$ such that $f''(t_1)=0$ is the union of a finite number of distinct points and non-intersecting intervals, and 
\item If $f''(t_1)\equiv0$ on an interval, then $f'(t_1)$ is irrational on that interval. 
\end{enumerate}
An obvious modification needs to be made for pieces where $\tilde v_1'(\vec t)\not=0$, $t\in\Gamma_t$.
\end{enumerate}
}

Using \eqref{loc-v2-crt}, \eqref{G-val}, define
\be\label{g1g2}\begin{split}
g^{(1)}(\al,p):=&[2\pi \fj(\al)/\sqrt{\text{det}Q(\al)}](p-p_0(\al))_+,\\ 
g^{(2)}(\al,p):=&(p-p_0(\al))_+(G(\al,p-p_0(\al))-G(\al,0)).
\end{split}
\ee
Let $f_{\chi,\e}^{(j)}$ denote the result of inserting $\chi$ and $g^{(j)}$ into the discrete inversion formula \eqref{invf-discr-crt}, $j=1,2$. By linearity, $f_{\chi,\e}=f_{\chi,\e}^{(1)}+f_{\chi,\e}^{(2)}$. Similarly to \eqref{lead-t-int-crt}, \eqref{f1-lt-crt}:
\be\label{fext-lt-crt}
\begin{split}
f_{\chi,\e}^{(1)}(x_0)= &-\frac{1}{2\pi} \sum_i \frac{c_i}{\e} \frac{\fj(\al_i)}{\sqrt{\text{det}Q(\al_i)}} \psi\left(\frac{\al_i\cdot x_0}{\e}-\check p(\al_i),\frac{\al_i\cdot x_0-p_0(\al_i)}{\e}\right)\chi(\al_i).
\end{split}
\ee

Let $[t_1^{(\text{min})},t_1^{(\text{max})}]$ be the projection of $\Gamma_t$ onto the $t_1$-axis. Partition the projection onto the intervals of length $\Delta t_1$: $[t_1^{(m)},t_1^{(m+1)}]$, 
$t_1^{(m+1)}-t_1^{(m)}=\Delta t_1$. 
Pick one of the subintervals, e.g. the $m$-th. Let $t_1^*$ be its center, and set $\vec t^*:=(t_1^*,A(t_1^*))$, $\al^*:=H(\vec t^*)$. For convenience, denote $\bar i_j:=i_j+r_j$, $j=1,2$, and let $\mathcal I_m$ be the set of indices $\vec i$ such that $\e\bar i_1\in [t_1^{(m)},t_1^{(m+1)}]$. Rewrite the part of the sum in \eqref{fext-lt-crt} that corresponds to the indices $\vec i\in \mathcal I_m$:
\be\label{part-sum-v2}
\begin{split}
J_m(\e):&=\sum_{\vec i\in \mathcal I_m} \frac{c_{\vec i}}{\e} \frac{\fj(\al_{\vec i})}{\sqrt{\text{det}Q(\al_{\vec i})}}
\psi\left(\frac{q(\al_{\vec i})}{\e}-\check p(\al_{\vec i}),\frac{v(\al_{\vec i})}{\e}\right)\chi(\al_{\vec i})\\
&=\frac{c^*}{\e} \frac{\fj(\al^*)}{\sqrt{\text{det}Q(\al^*)}}\chi(\al^*)\sum_{\vec i\in \mathcal I_m} 
\psi\left(\frac{q(\al_{\vec i})}{\e}-\check p(\al^*),\frac{v(\al_{\vec i})}{\e}\right)+O((\dq)^2).
\end{split}
\ee
Here we have used that the sum is $O(\dq)$ uniformly as $\e\to0$, and that $c_{\vec i}=c^*(1+O(\dq))$ for all $\vec i$ such that $\vec i\in \mathcal I_m$ and $\psi(\cdot,\cdot)\not=0$, where $c^*$ is the area of the tessellation piece containing $\al^*$. By \eqref{psi-loc}, there are $O(\e \dq)/\e^2$ nonzero terms in the sum. Combined with \eqref{wghts}, we get that the sum is indeed of order $O(\dq)$.

By construction,
\be\label{part-sum-v3}
\sum_{\vec i\in \mathcal I_m} 
\psi\left(\frac{q(\al_{\vec i})}{\e}-\check p(\al^*),\frac{v(\al_{\vec i})}{\e}\right)
=\sum_{\vec i\in \mathcal I_m} 
\psi\left(\frac{\tilde q(\e \bar i_1,\e \bar i_2)}{\e}-\check p(\al^*),\frac{\tilde v(\e \bar i_1,\e \bar i_2)}{\e}\right).
\ee
The assumption $\tilde v_2'(\vec t)\not=0$, $\vec t\in\Gamma_t$, implies that the lines $t_1=\e \bar i_1,t_2\in\br$, $\e \bar i_1\in [t_1^{(m)},t_1^{(m+1)}]$ intersect $\Gamma_t$ transversely. Then
\be\label{vexp}\begin{split}
\frac{\tilde v(\e \bar i_1,\e \bar i_2)}{\e}&=\tilde v_2'(\e \bar i_1,A(\e \bar i_1))\left(\bar i_2-\frac{A(\e \bar i_1)}{\e}\right)+O(\e)\\
&=c_v^*\left(\bar i_2-\frac{A(\e \bar i_1)}{\e}\right)+O(\dq)+O(\e),\ c_v^*:=\tilde v_2'(t_1^*,A(t_1^*)),
\end{split}
\ee
where we have used that $|\bar i_2-(A(\e \bar i_1)/\e)|$ is bounded whenever $\psi\not=0$, and $|\e \bar i_1-t_1^*|=O(\dq)$. There are two big-$O$ terms in \eqref{vexp}, because $\e$ and $\dq$ are independent variables. Similarly,
\be\label{tq-transf-v2}\begin{split}
\tilde q(\e \bar i_1,\e \bar i_2)&=\tilde q_1(\e \bar i_1)+c_q^*(\e \bar i_2-A(\e \bar i_1))+O(\e\dq)+O(\e^2),\\ 
\tilde q_1(t_1)&:=\tilde q(t_1,A(t_1)),\ c_q^*:=\tilde q'_2(t_1^*,A(t_1^*)).
\end{split}
\ee
By \eqref{psi-per}, the non-zero terms in \eqref{part-sum-v3} can be written in the form
\be\label{terms}\begin{split}
&\psi\left(\frac{\tilde q(\e \bar i_1,\e \bar i_2)}{\e}-\cp(\al^*),\frac{\tilde v(\e \bar i_1,\e \bar i_2)}{\e}\right)\\
&=\psi_1\left(\frac{\tilde q_1(\e \bar i_1)}{\e}-\cp(\al^*)+c_q^*\left(\bar i_2-\frac{A(\e \bar i_1)}{\e}\right),\bar i_2-\frac{A(\e \bar i_1)}{\e}\right)+O(\dq)+O(\e)\\
&=\psi_1\left(\left\{\frac{\tilde q_1(\e \bar i_1)}{\e}-\cp(\al^*)+c_q^*\left(\bar j_2-\left\{\frac{A(\e \bar i_1)}{\e}\right\}\right)\right\},\bar j_2-\left\{\frac{A(\e \bar i_1)}{\e}\right\}\right)\\
&\qquad+O(\dq)+O(\e),
\end{split}
\ee
where we denoted
\be\label{j2def}
\psi_1(q,\tau):=\psi(q,c_v^*\tau),\ 
\bar j_2:=\bar i_2-\left[\frac{A(\e \bar i_1)}{\e}\right].
\ee
Because of \eqref{psi-loc}, there are finitely many values of $\bar j_2$ for which $\psi$ in \eqref{terms} is non-zero.

For a fixed value of $\bar j_2$, the effect of the term $-\cp(\al^*)+c_q^*\bar j_2$ in the first argument of $\psi_1$ is just a shift that does not affect the u.d. property. So it can be ignored if we show that the points 
\be\label{stp1}
\left(\left\{\frac{\tilde q_1(\e \bar i_1)}{\e}-c_q^*\left\{\frac{A(\e \bar i_1)}{\e}\right\}\right\},-\left\{\frac{A(\e \bar i_1)}{\e}\right\}\right)
\ee
are u.d. in the square $[0,1]\times [-1,0]$ as $\e\to0$. Corollary~\ref{cor:udpts-2d-e} and Assumption \eqref{zeromeas} imply that the set
\be\label{stp2}
\left(\left\{\frac{\tilde q_1(\e \bar i_1)}{\e}\right\},-\left\{\frac{A(\e \bar i_1)}{\e}\right\}\right)
\ee
is u.d. in $[0,1]\times [-1,0]$ as $\e\to0$. Then, Lemma~\ref{lem:xpy} implies the desired result.


Pick one value of $\bar j_2$. The arguments of $\psi_1$ on the last line in \eqref{terms} are u.d. over the rectangle $[0,1]\times[\bar j_2-1,\bar j_2]$ with area $1$. There are $N_1(\e):=\dq/\e$  distinct values of $i_1$ such that $\e \bar i_1\in [t_1^{(m)},t_1^{(m+1)}]$. Thus, the area per each value is $1/N_1(\e)$. Summation over $\bar j_2$ just shifts this rectangle along the second variable. These rectangles are stacked precisely one on top of the other without overlap. Hence
\be\label{sum-lim}\begin{split}
\lim_{\e\to0}\frac1{N_1(\e)}\sum_{\vec i\in \mathcal I_m} 
\psi\left(\frac{q(\al_{\vec i})}{\e}-\cp(\al_{\vec i}),\frac{v(\al_{\vec i})}{\e}\right)
&=\int_{\br}\int_{0}^{1} \psi_1(q,\tau)dq d\tau+O(\dq)\\
&=\frac1{|c_v^*|}\int_{\br}\int_{0}^{1} \psi(q,v)dqdv+O(\dq).
\end{split}
\ee
Using \eqref{int-psi-st1} and then A5 gives:
\be\label{psi-crt-int}
\int_{\br}\int_0^1 \psi(q,v)dq dv=\int_{\br}\varphi(v)dv=1.
\ee
From \eqref{part-sum-v2}, \eqref{part-sum-v3}, \eqref{terms}, \eqref{sum-lim}, and \eqref{psi-crt-int} it follows that we have to evaluate the following expression:
\be\label{kappa-def}
\frac{c^*}{\e}\frac{N_1(\e)}{|c_v^*|}=\frac{c^*}{\e^2}\frac{\Delta t_1}{|c_v^*|}=\left(\dt\left.\left(\frac{\pa\al}{\pa\vec t}\right)\right|_{\vec t=\vec t^*}+O(\e)\right)\frac{\dq}{|c_v^*|}.
\ee
Here we have used that the elementary area in the $t_1,t_2$-coordinates is $\e^2$, and the area of the corresponding piece on $S^2$ is $c^*$. 
Combining the above results gives
\be\label{part-sum-st3}
\lim_{\e\to0}J_m(\e)=\chi(\al^*)\frac{\fj(\al^*)}{\sqrt{\text{det}Q(\al^*)}}\dt\left.\left(\frac{\pa\al}{\pa\vec t}\right)\right|_{\vec t=\vec t^*}\frac{\dq}{|c_v^*|}+O(\dq^2).
\ee
Next we sum $J_m(\e)$ for all $m$. Since $\dq>0$ can be arbitrarily small, we conclude from \eqref{part-sum-st3} by using \eqref{vexp} that  
\be\label{lead-trm-crt}\begin{split}
\lim_{\e\to0} f_{\chi,\e}^{(1)}(x_0) &= -\frac1{2\pi}\int_{t_1^{(\text{min})}}^{t_1^{(\text{max})}}\chi(\al)\frac{\fj(\al)}{\sqrt{\text{det}Q(\al)}}\dt\left(\frac{\pa\al}{\pa\vec t}\right)\frac1{|\tilde v_2'(t_1,A(t_1))|}dt_1\\
&=-\frac1{2\pi}\int_{U_t}\chi(\al)\frac{\fj(\al)}{\sqrt{\text{det}Q(\al)}}\dt\left(\frac{\pa\al}{\pa\vec t}\right)\frac{\de(t_2-A(t_1))}{|\tilde v_2'(t_1,A(t_1))|}d\vec t\\
&=-\frac1{2\pi}\int_{U_t}\chi(\al)\frac{\fj(\al)}{\sqrt{\text{det}Q(\al)}}\dt\left(\frac{\pa\al}{\pa\vec t}\right)\de(\tilde v(\vec t))d\vec t\\
&= -\frac1{2\pi}\int_{\Omega}\chi(\al)\frac{\fj(\al)}{\sqrt{\text{det}Q(\al)}} \de(\al\cdot x_0-p_0(\al))d\al.
\end{split}
\ee

The result in \eqref{lead-trm-crt} is, of course, expected, since this is what one gets in the continuous case by substituting $\chi(\al)$ and $g^{(1)}(\al,p)$ (cf. \eqref{g1g2}) into \eqref{invf-cont-crt}. The same result for $g^{(2)}(\al,p)$, which leads to $f_{\chi,\e}^{(2)}$, follows immediately, since the second order derivative or $g^{(2)}(\al,p)$ with respect to $p$ is piecewise-smooth and bounded on $S^2$. In this case the contribution of the terms in an $\e$-neighborhood of $\Gamma$ goes to zero as $\e\to0$.

Next we discuss how the above derivation changes when $\xh$ replaces $x_0$ in \eqref{fext-lt-crt}. Clearly, it is sufficient to consider only the leading singular term $g^{(1)}(\al,p)$. The analogue of \eqref{new-vars} becomes
\be\label{new-vars-xh}
q^{(h)}:=\al\cdot x_0+(\al\cdot\Theta_0)\e h,\ v^{(h)}:=\al\cdot x_0+(\al\cdot\Theta_0)\e h-p_0(\al).
\ee
From \eqref{tq-transf-v2} and \eqref{new-vars-xh}, we find
\be\label{tq-transf-new}\begin{split}
\tilde q^{(h)}(\e \bar i_1,\e \bar i_2)=&\tilde q(\e \bar i_1,\e \bar i_2)+(\al(\e \bar i_1,\e \bar i_2)\cdot\Theta_0)\e h\\
=&\tilde q(\e \bar i_1,\e \bar i_2)+(\al^*\cdot\Theta_0)\e h+O(\e\dq),\\
\tilde v^{(h)}(\e \bar i_1,\e \bar i_2)=&\tilde v(\e \bar i_1,\e \bar i_2)+(\al^*\cdot\Theta_0)\e h+O(\e\dq).
\end{split}
\ee
Hence the analogue of \eqref{terms} becomes:
\be\label{part-sum-h}
\begin{split}
&\psi\left(\frac{\tilde q^{(h)}(\e \bar i_1,\e \bar i_2)}{\e}-\cp(\al^*),\frac{\tilde v^{(h)}(\e \bar i_1,\e \bar i_2)}{\e}\right)\\
&=\psi_1\left(\frac{\tilde q_1(\e \bar i_1)}{\e}+(\al^*\cdot\Theta_0)h-\cp(\al^*)+c_q^*\left(\bar i_2-\frac{A(\e \bar i_1)}{\e}\right)\right.,\\
&\hspace{2cm}\left.\bar i_2-\frac{A(\e \bar i_1)}{\e}+\frac{(\al^*\cdot\Theta_0)h}{c_v^*}\right)+O(\dq).
\end{split}
\ee
The new arguments of $\psi_1$ are obtained from the old ones by a constant shift. Clearly, this does not affect the averaging argument in \eqref{j2def}--\eqref{part-sum-st3}. The interval $[t_1^{(\text{min})},t_1^{(\text{max})}]$ needs to be enlarged slightly to account for the dependence of $\Gamma_t$ on $h$. The derivation works the same way, and we get the same result as in \eqref{lead-trm-crt} even if $\xh$ replaces $x_0$ on the left side of \eqref{lead-trm-crt}. Thus, we proved the following result.

\begin{theorem}\label{thm:remote} Pick a generic point $x_0\not\in S$. For any $\Theta_0\in S^2$, $h$ confined to a bounded set, and any $\chi\in C^\infty(S^2)$ one has 
\be\label{res-remote}
\lim_{\e\to0}f_{\chi,\e}(x_0+\e h\Theta_0)=-\frac1{4\pi^2}\int_{S_+^2}\left. \chi(\al)(\pa/\pa p)^2 g(\al,p)\right|_{p=\al\cdot x} d\al.
\ee
\end{theorem}

Combining \eqref{f1-st4-crt}, \eqref{int-psi-st1}, \eqref{f3-crt-st2}, \eqref{part2-lim}, \eqref{twoterms}, and \eqref{res-remote}, we finish the proof of Theorem~\ref{thm:main}.

\section{Numerical experiments}\label{numerix}

We start by constructing an interpolation kernel $\ik$ with the required properties. We use the method of \cite{btu03}. In \cite{btu03} the authors construct piecewise-polynomial kernels $\kod$, which are characterized by a quadruple of numbers $\{N,W,R,L\}$. Here 
\begin{enumerate}
\item $N$ is the maximal degree of  polynomial pieces that make up $\kod$, 
\item $W$ is the support: $\text{supp}(\kod)=[0,W]$, 
\item $R$ is regularity, i.e. $\kod$ is $R$ times continuously differentiable, and
\item $L$ is order, i.e. the approximation error behaves like $O(\e^L)$, where $\e$ is step-size.
\end{enumerate}
Additional requirements, e.g. that $\kod$ be symmetric and interpolating can be imposed too. 

Assumption A3 implies that $\kod\in C^2(\br)$. Since $\kod$ is piecewise-polynomial, $\kod'''$ is automatically piecewise continuous and bounded, so Assumptions A3 and A4 are satisfied if $R=2$. As is easy to see, Assumption A1 is satisfied if $L=3$. Assumption A2 means that $W < \infty$. Hence we should have $R=2$, $L=3$, and $W < \infty$. The construction in \cite{btu03} involves a number of free parameters (degrees of freedom) $P$. In our case there are
\be\label{deg-fred}
P=(N-R)(W-L)+L-R-1=(N-2)(W-3)
\ee
degrees of freedom. The requirement that $\kod$ be interpolating: $\ik((W/2)+n)=\delta_n$ imposes $W$ (respectively, $W-1$) additional conditions if $W$ is odd (respectively, even). Consequently, $(N-2)(W-3)\ge W$ (resp., $W-1$) if $W$ is odd (resp., even). Hence, $N\ge 4$ and $W\ge 4$. Choosing $N=4$, gives the minimal acceptable value $W=6$. Thus, $\kod$ can be selected in the class $\{N=4,W=6,R=2,L=3\}$. 

Imposing the condition that $\kod$ be symmetric and using equations (20), (22), and (23) in \cite{btu03} implies that $\kod$ can be represented in the form
\be\label{1st-form}
\varphi(t)=a_1(B_3(t)+B_3(6-t)) +a_2B_3(t-1)+a_3(B_4(t)+B_4(t-1)),
\ee  
where $a_{1,2,3}$ are some constants, and $B_n$ denotes the cardinal B-spline of degree $n$ supported on $[0,n+1]$. Finding $a_{1,2,3}$ by solving the system $\ik((W/2)+n)=\delta_n$ and simplifying gives the final expression
\be\label{final-form}
\ik(t)=0.5(B_3(t)+B_3(t-2))+4B_3(t-1)-2(B_4(t)+B_4(t-1)).
\ee  
Using the properties of B-splines, we see that $\int_{\br}\ik(t)dt=1$.

The test object consists of two balls with centers $c_1=(0,0,-5)$ and $c_2=(-5.52,0,-7.36)$. Each ball has radius $R=4$ and uniform density 1. The point on the boundary $x_0$, in a neighborhood of which we study resolution, is given by
\be\label{point}
x_0=c_1-R\Theta_0,\ \Theta_0=-(\cos(0.7\pi)\sin(0.2\pi),\sin(0.7\pi)\sin(0.2\pi),\cos(0.2\pi)).
\ee
In agreement with our convention, $\Theta_0$ points into the interior of the first ball. The second ball is chosen in such a way that the plane through $x_0$ and normal to $\Theta_0$ is tangent to its boundary. This way we demonstrate that remote singularities do not contribute to edge response.

To simulate discrete data, Radon transform is computed at the points 
\be\label{datapts}\begin{split}
\al_{\vec i}&=(\cos\theta_{i_1}\sin\ga_{i_2},\sin\theta_{i_1}\sin\ga_{i_2},\cos\ga_{i_2}),\
p_j=-10+0.04j,0\le j\le 500,\\
\theta_{i_1}&=\frac{2\pi}{500}i_1,\ 0\le i_1\le 500,\ \ga_{i_2}=\frac{\pi}{500}i_2,\ 0\le i_2\le 500.
\end{split}
\ee
Thus, $\e=\Delta p=0.04$. The predicted response is computed using \eqref{main-res}, where $f_0=\rho_0=1$. The results are shown in Figure~\ref{fig:04}. The $x$-axis in the figure shows the $h$-values. The $y$-axis shows the reconstructed values at the points $x_h=x_0+\e h\Theta_0$ (cf. \eqref{rec-pt-crt}) as well as the predicted transition curve (edge response). The results show a good match.

\begin{figure}[h]
{\centerline{\epsfig{file=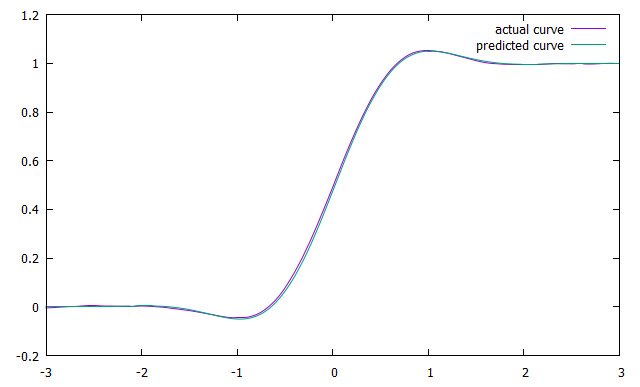, width=10cm}}
}
\caption{Comparison of the predicted and actual transition curves for $\e=0.04$.}
\label{fig:04}
\end{figure}

Next we illustrate the situation where condition \eqref{Hgrad} does not hold and Theorem~\ref{thm:main} does not hold either. The test object consists of a single ball of radius 1 with center $c=(0,0,0)$ in the first experiment, and center $c=(0,0,1)$ - in the second experiment. The point on the boundary $x_0$ is now given by
\be\label{point-v2}
x_0=c-R\Theta_0,\ \Theta_0=(-1,0,0).
\ee
As is easily seen, in both cases the components of the vector $\mathcal H_0'$ (cf. \eqref{hprime}) are integers, so $x_0$ is not generic. The results are shown in Figure~\ref{fig:alt}. The plot on the left is for the case $c=(0,0,0)$, and on the right - for the case $c=(0,0,1)$. The match between the predicted and actual transition behaviors is no longer accurate.

\begin{figure}[h]
{\centerline{
\hbox{
\epsfig{file=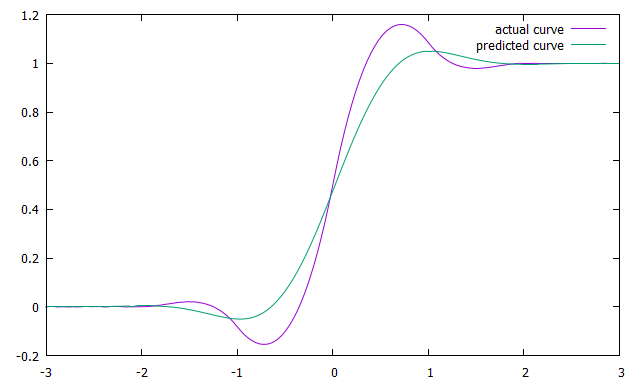, width=6.4cm}}
{\epsfig{file=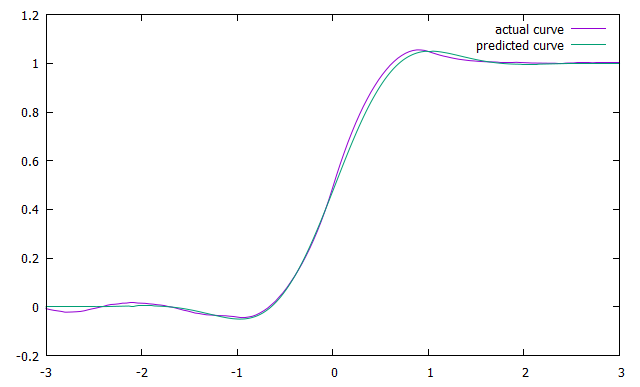, width=6.4cm}}
}}
\caption{Comparison of the predicted and actual transition curves for $\e=0.04$. Left panel: the center is at $c=(0,0,0)$, right panel: the center is at $c=(0,0,1)$. The match is not accurate, because $x_0$ is not generic.}
\label{fig:alt}
\end{figure}

\appendix
\section{Auxiliary results}\label{sex:auxres}

\begin{lemma}\label{lem:udpts-e} Given an interval $[a,b]$, consider a real-valued function $f\in C^2([a,b])$ such that either (1) $f''\equiv 0$ on $[a,b]$ and $f'$ equals to an irrational number, or (2) 
$f''\not=0$ on $[a,b]$, and 
\be\label{incl-e}
n< f'(a),f'(b) < n+1
\ee
for some $n\in\mathbb Z$. Then the points $\{f(\e i)/\e\}$, $\e i\in[a,b]$, are u.d. as $\e\to0$. 
\end{lemma}
\begin{proof} Case (1) is well-known (see \cite{KN_06}), so we only consider case (2). Using Weyl's criterion, we wish to show that 
\be\label{wlcr-e}
J:=\e\sum_{\e i\in[a,b]}e(Mf(\e i)/\e)\to 0,\ \e\to0,
\ee
for all integers $M\not=0$, where $e(t):=\exp(2\pi i t)$. Let $M\ge 1$ be fixed in what follows. The case $M\le -1$ can be reduced to the case $M\ge 1$ by complex conjugation.

Without loss of generality, we may assume that $f''>0$ and $f'(a)<f'(b)$. Let $a_m$ be the points such that 
\be\label{derpart-e}\begin{split}
f'(a_m)=f'(a)+(f'(b)-f'(a))\frac mM,\ 0\le m\le M.
\end{split}
\ee
Suffice it to show that
\be\label{wlcrm-e}
J_m:=\e\sum_{\e i\in[a_m,a_{m+1}]}e(Mf(\e i)/\e)\to 0,\ \e\to0,
\ee
for each $m$,\ $0\le m<M$.

Partition $[a_m,a_{m+1}]$ into the intervals of length $\e^{2/3}$:
\be\label{intervs-e}
\Delta_k:=[t_k,t_{k+1}],\ t_k:=a_m+k\e^{2/3},\ 0\le k<\Delta \e^{-2/3},\ \Delta:=a_{m+1}-a_m.
\ee
By assumption, $f''(t)$ is bounded away from zero on $[a,b]$, so each subinterval $\Delta_k$ contains $O\left(\e^{-1/3}\right)$ points from the original set $\e i$. 

First, we show that
\be\label{frst-est-e}\begin{split}
I:&=\left|\e^{1/3}\sum_{\e i\in \Delta_k}e\left(M\frac{f(\e i)}\e\right)
-\e^{1/3}\sum_{\e i\in \Delta_k}e\left(M\frac{f(t_k)+f'(t_k)(\e i-t_k)}\e\right)\right|
\to 0,\\ 
\e&\to0.
\end{split}
\ee
Since
\be\label{taylor2-e}
f(\e i)=f(t_k)+f'(t_k)(\e i-t_k)+O(\e^{4/3}),\ \e i\in\Delta_k,
\ee
it follows from \eqref{frst-est-e}
\be\label{Ismall-e}
I\le \e^{1/3}\sum_{\e i\in \Delta_k}\left|e\left(O\left(\e^{1/3}\right)\right)-1 \right|=
O\left(\e^{1/3}\right).
\ee
The big-$O$ terms in \eqref{taylor2-e} and \eqref{Ismall-e} are uniform in $k$, because $\max_{t\in[a,b]}f''(t)<\infty$ uniformly as $\e\to0$.

To prove \eqref{wlcrm-e}, partition the sum and then use \eqref{frst-est-e}, \eqref{Ismall-e} to obtain an estimate:
\be\label{Jest1-e}\begin{split}
|J_m|&\le \e \sum_{k=0}^{\Delta \e^{-2/3}}\left|\sum_{\e i\in \Delta_k}e\left(M\frac{f(t_k)+f'(t_k)(\e i-t_k)}\e\right)\right|+O\left(\e^{1/3}\right)\\
&= \e \sum_{k=0}^{\Delta \e^{-2/3}}\left|\sum_{\e i\in \Delta_k}e\left(Mf'(t_k)i\right)\right|+O\left(\e^{1/3}\right)
=:J_m^{(l)}+J_m^{(r)}+O\left(\e^{1/3}\right).
\end{split}
\ee
In $J_m^{(l)}$, the sum is over $k$ such that 
\be\label{left-e}
f'(t_k)-f'(a_m)<f'(a_{m+1})-f'(t_k).
\ee
In $J_m^{(r)}$, the sum is over $k$ such that 
\be\label{right-e}
f'(t_k)-f'(a_m)>f'(a_{m+1})-f'(t_k).
\ee
Each of the sums is over $O(\e^{-2/3})$ consequtive values of $k$. In particular, the number of terms in each sum is bounded from below by $c\e^{-2/3}$ for all $\e>0$ sufficiently small, where $c>0$ is some constant.

Consider $J_m^{(l)}$. We use the inequality
\be\label{keyineq-e}
\left| \sum_{i=1}^Ne(r i)\right|\le \min(N,1/(2\Vert r \Vert)),
\ee
where $\Vert r \Vert$ denotes the distance from $r$ to the nearest integer. By \eqref{incl-e}, \eqref{derpart-e}, \eqref{intervs-e}, and \eqref{left-e},
\be\label{disttoint-e}
\Vert M f'(t_k)\Vert\ge 
| M f'(t_k)-Mf'(a_m)|\ge
M\de k\e^{2/3},\ \de:=\min_{t\in[a,b]} f''(t),
\ee
for all $k$ used in $J_m^{(l)}$. 
Using \eqref{keyineq-e} in \eqref{Jest1-e} with
\be\label{consts-e}
r=Mf'(t_k),\ N=O(\e^{-1/3}),
\ee
gives
\be\label{midest-e}
\left|\sum_{\e i\in \Delta_k}e\left(Mf'(t_k)i\right)\right|\le\min\left(O(\e^{-1/3}),\frac{1}{2M\de k\e^{2/3}}\right).
\ee
This implies,
\be\label{J1est-e}\begin{split}
J_m^{(l)}&\le \e
\left( \sum_{0\le k<O(\e^{-1/3})}O(\e^{-1/3})+\sum_{O(\e^{-1/3})\le k<O(\e^{-2/3})} \frac{O(1)}{k\e^{2/3}}\right)\\
&=O\left(\e^{1/3}\ln(1/\e)\right)\to0.
\end{split}
\ee
The quantity $J_m^{(r)}$ can be estimated in a similar fashion (using \eqref{right-e} instead \eqref{left-e}), and the lemma is proven.
\end{proof}

\begin{corollary}\label{cor:udpts-v2-e} Given an interval $[a,b]$, pick a real-valued function $f\in C^2([a,b])$. Suppose that
\begin{enumerate}
\item The set $\{x\in[a,b]:\, f''(x)=0\}$ is the union of a finite number of distinct points and non-intersecting intervals, and 
\item If $f''(x)\equiv0$ on an interval, then $f'(x)$ is irrational on that interval. 
\end{enumerate}
Then the points $\{f(\e i)/\e\}$, $\e i\in[a,b]$, are u.d. as $\e\to0$.
\end{corollary}
\begin{proof} Let $X$ be the set consisting of the isolated points where $f''(x)=0$ and of the endpoints of the intervals where $f''(x)\equiv0$. By assumption, $X$ contains finitely many points. Let $X_\de$ denote the $\de$-neighborhood of $X$. Clearly, $[a,b]\setminus X_\de$ is a finite union of non-overlapping intervals $[a_i,b_i]$ such that (1) $|b-a|-\sum_i|b_i-a_i|$ can be made as small as we like by choosing $\de>0$ small, and (2) each of the intervals $[a_i,b_i]$ satisfies the assumptions of Lemma~\ref{lem:udpts-e}. Thus the points $\{f(\e i)/\e\}$, $\e i\in \cup_i [a_i,b_i]$, are u.d. as $\e\to0$. The fraction of the remaining points can be made arbitrarily small, which proves the desired assertion.
\end{proof}

\begin{corollary}\label{cor:udpts-2d-e} Given an interval $[a,b]$, pick two functions $g,h$. Suppose the linear combination $f:=M_1g+M_2h$ satisfies the assumptions of Corollary~\ref{cor:udpts-v2-e} for any pair $M_1,M_2\in\mathbb Z$, $M_1^2+M_2^2>0$. Then the set  of points $(\{g(\e i)/\e\},\{h(\e i)/\e\})$, $\e i\in[a,b]$, is u.d. in $[0,1]\times[0,1]$ as $\e\to0$. 
\end{corollary}
\begin{proof} The assertion follows immediately from \cite{KN_06}, Theorem 6.3, and Corollary~\ref{cor:udpts-v2-e}.
\end{proof}

\begin{lemma}\label{lem:xpy} Let the set $(x_\e(i),y_\e(i))$, $i\in I_\e$, be u.d. in the square $[0,1]\times[0,1]$ as $\e\to0$. Here $I_\e$ is some set of indices that depends on $\e$. Then the set $(\{x_\e(i)+a y_\e(i)\},y_\e(i))$, $i\in I_\e$, is also u.d. in $[0,1]\times[0,1]$ as $\e\to0$ for any $a\in\br$.
\end{lemma}
\begin{proof} Pick some $r_0,y_0,\Delta r,\Delta y>0$ so that $r_0+\Delta r, y_0+\Delta y\in[0,1)$. Consider the domain
\be\label{domain} 
D=\{(x,y)\subset [0,1)\times[0,1):\ \{x+y\}\in [r_0,r_0+\Delta r), y\in [y_0,y_0+\Delta y)\}. 
\ee 
As is seen from elementary geometry (proof by picture), the area of $D$ is independent of $r_0,y_0$ and equals $\Delta r\Delta y$. Since $(x_\e(i),y_\e(i))$ is u.d. in $[0,1]\times[0,1]$, the desired assertion follows.
\end{proof}

\bibliographystyle{plain}
\bibliography{bibliogr_A-K,bibliogr_L-Z}
\end{document}